\documentclass[10pt, one side,emlines]{amsart}
\usepackage{amssymb,latexsym,xy,eucal,mathrsfs,tikz}
\textwidth=15.2cm \textheight=24cm \theoremstyle{plain}
\newtheorem{lem}{Lemma}[section]

\newtheorem{thm}[lem]{Theorem}

\newtheorem{cor}[lem]{Corollary}
\theoremstyle{definition}

\newtheorem{rem}[lem]{Remark}

\newtheorem{defn}[lem]{Definition}

\newtheorem*{st}{\textnormal{\textbf{Structure Theorem}}}
\newtheorem*{nt}{\textnormal{\textbf{Notation}}}
\theoremstyle{remark}

\numberwithin{equation}{section}  \voffset
-55truept \hoffset -15truept

\begin{document}
\baselineskip 15truept

\hfill{\it Accepted for publication in Mathematica Slovaca}\vspace{.21in}

\title{Zero-divisor graphs of lower dismantlable lattices-I}
\date{}

\dedicatory{Dedicated to
Professor N. K. Thakare on his $77^{th}$ birthday}

\author[Avinash Patil, B. N. Waphare, V. V. Joshi \and H. Y. Pourali] %
{Avinash Patil*, B. N. Waphare**, V. V. Joshi** \and H. Y. Pourali**}

\newcommand{\acr}{\newline\indent}

\address{\llap{*\,}Department of Mathematics\acr
                   Garware College of Commerce\acr
                   Karve Road, Pune-411004\acr
                   India.}
\email{avipmj@gmail.com}

\address{\llap{**\,}Department of Mathematics\acr
                  Savitribai Phule Pune Universitye\acr
                   Pune-411007\acr
                   India.}

\email{waphare@yahoo.com; bnwaph@math.unipune.ac.in}
\email{vvj@math.unipune.ac.in; vinayakjoshi111@yahoo.com}
\email{hosseinypourali@gmail.com} 
\begin{abstract}In this paper, we study the zero-divisor graphs
of a subclass of dismantlable lattices. These graphs are characterized in terms of the non-ancestor graphs of rooted
trees.\end{abstract} \maketitle \noindent
{\bf MSC(2010):} {Primary $05$C$25$, Secondary $05$C$75$}.

\section{Introduction}
Beck \cite{2} introduced the concept  of zero-divisor graph of a
commutative ring $R$ with unity as follows. Let $G$ be a simple
graph whose vertices are the elements of $R$ and two vertices $x$
and $y$ are adjacent if $xy = 0$. The graph $G$ is known as the
\textit{zero-divisor graph} of $R$. He was mainly interested in the
coloring of this graph. This concept is well studied in
algebraic structures such as rings, semigroups, lattices,
semilattices as well as in ordered structures such as posets and
qosets; see Anderson et al. \cite{DP}, Alizadeh et al. \cite{1},
LaGrange \cite{11,12}, Lu and Wu \cite{13}, Joshi and Khiste
\cite{8}, Nimbhorkar et al. \cite{15}, Hala\v {s} and Jukl
\cite{4}, Joshi \cite{7}, Joshi, Waphare and Pourali \cite{9, jwp} and
Hala\v {s} and L\"{a}nger \cite{5}.

A graph is called \textit{realizable as zero-divisor graph} if it
is isomorphic to the zero-divisor graph of an algebraic structure
or an ordered structure.  In \cite{12}, LaGrange characterized
simple graphs which are realizable as the zero-divisor graphs of
Boolean rings and in \cite{13}, Lu and Wu gave a class of graphs
that are realizable as the zero-divisor graphs of posets.
Recently, Joshi and Khiste \cite{8} extended the result of
LaGarange \cite{12} by characterizing simple graphs which are
realizable as zero-divisor graphs of Boolean posets.

In this paper, we provide a class of graphs, namely the
non-ancestor graphs of rooted trees, that are realizable as
zero-divisor graphs of lower dismantlable lattices. In fact we
prove:
\begin{thm}\label{t1}For a simple undirected graph $G$, the following statements are equivalent.
\begin{enumerate}
\item[$(a)$] $G\in \mathcal{G_T}$, the class of non-ancestor
graphs of rooted trees. \item[$(b)$] $G=G_{\{0\}}(L)$ for
some lower dismantlable lattice $L$ with the greatest element $1$
as a join-reducible element. \item[$(c)$] $G$ is the
incomparability graph of $(L\backslash \{0,1\},\leq) $ for some
lower dismantlable lattice $L$ with the greatest element $1$ as a
join-reducible element.
\end{enumerate}
\end{thm}

Rival \cite{17} introduced dismantlable lattices to study the
combinatorial properties of doubly irreducible elements. By
dismantlable lattice, we mean a lattice which can be completely
``dismantled" by removing one element at each stage. Kelly and Rival
\cite{10} characterized dismantlable lattices by means of crowns,
whereas Thakare, Pawar and Waphare \cite{18} gave a structure
theorem for dismantlable lattices using adjunct operation.

Now we begin with the necessary definitions and terminology.
\begin{defn}
A nonzero element $p$ of a lattice $L$ with 0 is an {\it atom} if $0\prec p$ (by
$a\prec b$, we mean there is no $c$ such that $a<c<b$). Dually, a nonzero element $d$ of a lattice $L$ with 1 is a \textit{dual atom} if $d\prec 1$.
\end{defn}

\begin{defn}[Definition 2.1, Thakare et al. \cite{18}]\label{d1}
 If $L_1$ and $L_2$ are two disjoint finite lattices and $\; (a, b)\; $ is a pair
of elements in $L_1$ such that $a < b$ and $a \not\prec b$. Define the
partial order $\leq$ on $L = L_1 \cup L_2$ with respect to
the pair $(a,b)$ as follows.

$x \leq y$ in $L$ if

either $x,y \in L_1$ and $x \leq y$ in $L_1$;

or $x,y \in L_2$ and $x \leq y$ in $L_2$;

or $x \in L_1,$ $ y \in L_2$ and $x \leq a$ in $L_1$;

or $x \in L_2,$ $ y \in L_1$ and $b \leq y$ in $L_1$.

It is easy to see that $L$ is a lattice containing $L_1$ and $L_2$
as sublattices. The procedure of obtaining $L$ in this way is
called an {\it adjunct operation of $L_2$ to $L_1$}. The pair $(a,b)$
is called {\it an adjunct pair } and $L$ is an {\it adjunct }
of $L_2$ to $L_1$ with respect to the adjunct pair ($a,b$) and we
write $L = L_1 ]^b_a L_2$.
\end{defn}

We place the Hasse diagrams of $L_1$, $L_2$ side
by side in such a way that the greatest element $1_{L_2}$ of $L_2$
is at the lower position than $b$ and the least element $0_{L_2}$
of $L_2$ is at the higher position than $a$. Then add the
coverings $1_{L_2}\prec b$ and $ a\prec 0_{L_2}$, as shown in
Figure \ref{f1}, to obtain the Hasse diagram of $L=L_1]_a^bL_2$.
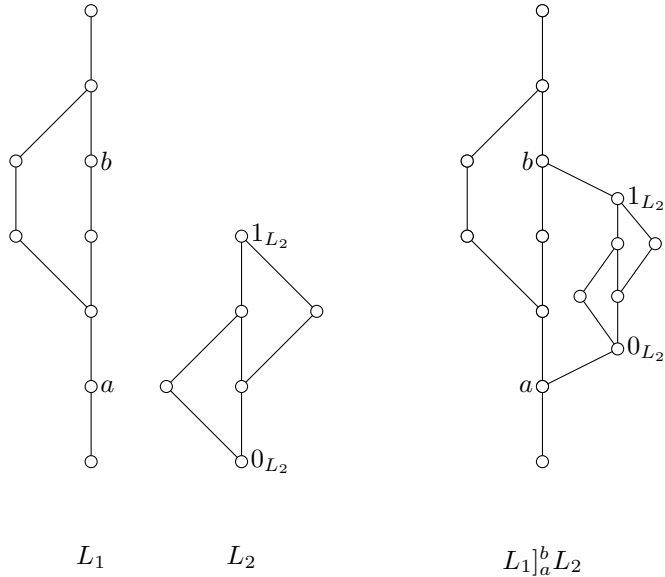
\begin{figure}[h]
\begin{tikzpicture}
\draw (0,0)--(0,6); \draw (0,2)--(-1,3)--(-1,4)--(0,5);
\draw[fill=white](0,0) circle(.08); \draw[fill=white](0,1)
circle(.08); \draw[fill=white](0,2) circle(.08);
\draw[fill=white](0,3) circle(.08); \draw[fill=white](0,4)
circle(.08); \draw[fill=white](0,5) circle(.08);
\draw[fill=white](0,6) circle(.08); \draw[fill=white](-1,3)
circle(.08); \draw[fill=white](-1,4) circle(.08); \node [right] at
(0,1){$a$}; \node [right] at (0,4){$b$}; \node [below] at
(0,-1){$L_1$};

\draw (2,0)--(2,3); \draw (2,0)--(1,1)--(2,2); \draw
(2,1)--(3,2)--(2,3); \draw[fill=white](2,0) circle(.08);
\draw[fill=white](2,1) circle(.08); \draw[fill=white](2,2)
circle(.08); \draw[fill=white](2,3) circle(.08);
\draw[fill=white](1,1) circle(.08); \draw[fill=white](3,2)
circle(.08); \node [below] at (2,-1){$L_2$}; \node [right] at
(2,0){$0_{\tiny L_2}$};\node [right] at (2,3){$1_{\tiny L_2}$};

\draw (6,0)--(6,6); \draw (6,2)--(5,3)--(5,4)--(6,5);
\draw[fill=white](6,0) circle(.08); \draw[fill=white](6,1)
circle(.08); \draw[fill=white](6,2) circle(.08);
\draw[fill=white](6,3) circle(.08); \draw[fill=white](6,4)
circle(.08); \draw[fill=white](6,5) circle(.08);
\draw[fill=white](6,6) circle(.08); \draw[fill=white](5,3)
circle(.08); \draw[fill=white](5,4) circle(.08);

\draw (6,1)--(7,1.5); \draw (6,4)--(7,3.5); \draw
(7,1.5)--(7,3.5); \draw (7,1.5)--(6.5,2.2)--(7,2.9); \draw
(7,2.2)--(7.5,2.9)--(7,3.5); \draw[fill=white](7,1.5) circle(.08);
\draw[fill=white](7,2.2) circle(.08); \draw[fill=white](7,2.9)
circle(.08); \draw[fill=white](7,3.5) circle(.08);
\draw[fill=white](6.5,2.2) circle(.08); \draw[fill=white](7.5,2.9)
circle(.08); \node [left] at (6,1){$a$}; \node [left] at
(6,4){$b$}; \draw[fill=white](6,1) circle(.08);
\draw[fill=white](6,2) circle(.08); \draw[fill=white](6,3)
circle(.08); \draw[fill=white](6,4) circle(.08);
\draw[fill=white](6,5) circle(.08); \draw[fill=white](6,6)
circle(.08); \draw[fill=white](5,3) circle(.08);
\draw[fill=white](5,4) circle(.08); \node [below] at (6,-1)
{$L_1]_a^bL_2$}; \node [right] at (7,1.5){$0_{\tiny L_2}$};\node
[right] at (7,3.5){$1_{\tiny L_2}$};
\end{tikzpicture}
\caption{Adjunct of two lattices $L_1$ and $L_2$}\label{f1}
\end{figure}

Clearly,  $|E(L)| = |E(L_1)| + |E(L_2)|+2$, where $E(L)$ is
nothing but edge set of $L$. This also implies that the adjunct
operation preserves all the covering relations of the individual
lattices $L_1$ and $L_2$. Also note that if $x,y\in L_2$, then $a\prec 0_{L_2}\leq x\wedge y$. Hence $x\wedge y\neq 0$ in $L=L_1]_a^bL_2$.

\pagebreak

\section{properties of zero-divisor graphs of dismantlable lattices}

Following Beck \cite{2}, Nimbhorkar et al. \cite{15}
introduced the concept of zero-divisor graph of meet-semilattices
with 0, which was further extended by Hala\v {s} and Jukl \cite{4}
to posets with 0. Recently, Joshi \cite{7} introduced the
zero-divisor graph with respect to an ideal $I$ of a poset with 0.
Note that his definition of zero-divisor graph coincides with the
definition of Lu and Wu \cite{13} when $I=\{0\}$.

A nonempty subset $I$ of a lattice $L$ is an {\it ideal} of $L$ if
$a,b\in I$ and $c\in L$ with $c\leq a$ implies $c\in I$ and $a\vee
b\in I$. An ideal $I\neq L$ is a {\it prime ideal} if $a\wedge
b\in I$ implies either $a\in I$ or $b\in  I$. A prime ideal $P$ of
a lattice $L$ is a {\it minimal prime ideal} if for any prime ideal
$Q$ we have $P\subseteq Q\subseteq L$ implies either $P=Q$ or
$Q=L$.

Now, we recall the definition of zero-divisor graph given by Joshi \cite{7} when the corresponding poset is a lattice and an
ideal $I=\{0\}$.
\begin{defn}[Definition 2.1, Joshi \cite{7}]\label{d2} Let $L$ be a lattice with the least element 0. We associate a simple
undirected graph $G_{\{0\}}(L)$ as
follows. The set of  vertices of $G_{\{0\}}(L)$ is\\
$V\left(G_{\{0\}}(L)\right)=\Bigl\{x \in L \setminus \{0\}~:~x\wedge y=0$ for some $y \in L\setminus \{0\} \Bigr\}$ and
distinct vertices $x, y$ are adjacent if and only if $x\wedge
y=0$. The graph $G_{\{0\}}(L)$ is called the {\it zero-divisor
graph of $L$}.
\end{defn}
The following Figure \ref{f2} illustrates the zero-divisor graph
$G_{\{0\}}(L)$ of the given lattice $L$.
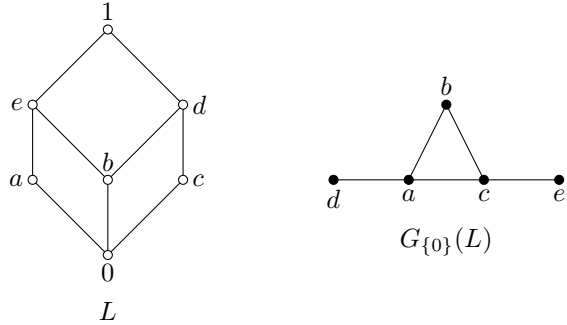
\begin{figure}[h]
\begin{tikzpicture}
\draw (1,0)--(1,1)--(0,2)--(0,1)--(1,0)--(2,1)--(2,2)--(1,1);\draw
(0,2)--(1,3)--(2,2);\draw [fill=white] (1,0) circle (.06); \draw
[fill=white] (1,1) circle (.06);\draw [fill=white] (0,2) circle
(.06);\draw [fill=white] (0,1) circle (.06);\draw [fill=white]
(2,1) circle (.06);\draw [fill=white] (2,2) circle (.06);\draw
[fill=white] (1,3) circle (.06);\node [below] at (1,0) {$0$};\node
[left] at (0,1) {$a$};\node [above] at (1,1) {$b$};\node [right]
at (2,1) {$c$};\node [right] at (2,2) {$d$};\node [left] at (0,2)
{$e$};\node [above] at (1,3) {$1$};\node [below] at (1,-.5) {$L$};
\draw(5,1)--(6,1); \draw(4,1)--(5,1)--(5.5,2)--(6,1)--(7,1);\draw
[fill=black] (5,1) circle (.06);\draw [fill=black] (4,1) circle
(.06);\draw [fill=black] (6,1) circle (.06);\draw [fill=black]
(7,1) circle (.06);\draw [fill=black] (5.5,2) circle (.06);\node
[below] at (4,1) {$d$};\node [below] at (5,1) {$a$};\node [below]
at (6,1) {$c$};\node [below] at (7,1) {$e$};\node [above] at
(5.5,2) {$b$};\node [below] at (5.5,.5) {$G_{\{0\}}(L)$};
\end{tikzpicture}
\caption{A lattice $L$ with its zero-divisor graph
$G_{\{0\}}(L)$}\label{f2}
\end{figure}

 Note that $\{0\}$ is a prime ideal in a lattice $L$ with 0 if and only if $V\left(G_{\{0\}}(L)\right)=\emptyset$.

 Now, we reveal the structure of zero-divisor graph of
 $L=L_1]_a^bL_2$ in terms of zero-divisor graphs of $L_1$ and
 $L_2$. For that purpose, we need the following definitions.
\begin{defn}\label{d3}Given two graphs $G_1$ and $G_2$, the {\it union} $G_1\cup G_2$
is the graph  with  $V(G_1\cup G_2)=V(G_1)\cup V(G_2)$ and
$E(G_1\cup G_2)=E(G_1)\cup E(G_2)$. The {\it join} of $G_1$ and
$G_2$, denoted by $G_1+G_2$ is the graph with  $V(G_1+
G_2)=V(G_1)\cup V(G_2)$, $E(G_1+ G_2)=E(G_1)\cup E(G_2)\cup J$,
where $J=\Bigl\{\{x_1,x_2\}\;|\; x_1\in V(G_1), x_2 \in
V(G_2)\Bigr\}$. The {\it null graph} on a set $S$ is the graph
whose vertex set is $S$ and edge set is the empty set, we denote
it by $N(S)$.\end{defn}

\textbf{Throughout this paper all the lattices are finite.}

The following result describes the zero-divisor graph of adjunct
of two lattices.

 \begin{thm}\label{t2} Let $L_1$ and $L_2$ be two lattices.  Put $L=L_1]_a ^b L_2$. Then the following statements are true.
 \begin{enumerate}
 \item[$(a)$] If $a\neq 0$ and $a\notin V\left(G_{\{0\}}(L_1)\right)$, then $G_{\{0\}}(L)= G_{\{0\}}(L_1)$.
   \item[$(b)$]  If $a\in V\big(G_{\{0\}}(L_1)\big)$, then $G_{\{0\}}(L)= G_{\{0\}}(L_1)\cup \big(G_a + N(L_2)\big)$,
   where \\ $G_a=\{x\in L_1 ~|~ a,x ~ \textnormal{are adjacent in } G_{\{0\}}(L_1)\}$.
 \item[$(c)$] If $a=0$, then $G_{\{0\}}(L)= G_{\{0\}}(L_1)\cup \Big( N\big(L_1^* \backslash [b)\big) + N(L_2)\Big)$, where
 $[b)=\{x\in L_1~:~ x\geq b\}$ is the principal dual ideal generated by $b$ in $L_1$ and $L_1^*= L_1\backslash \{0\}$.
 \end{enumerate}
 \end{thm}
\begin{proof}$(a)$ Let $a\neq 0$ and $a\notin V\left(G_{\{0\}}(L_1)\right)$. If $x\in V\left(G_{\{0\}}(L)\right)$ is adjacent to some $y\in L_2$,
then $x\wedge y=0$, clearly $x\notin L_2$. As $a\leq y$ in $L$, we get $a\wedge x
=0$. Hence $a\in  V\left(G_{\{0\}}(L_1)\right)$, a contradiction to
the fact that $a\notin V\left(G_{\{0\}}(L_1)\right)$. Hence no element
of $L_2$ is adjacent to any vertex of $G_{\{0\}}(L)$. Also, $G_{\{0\}}(L_1)$ is a subgraph of $G_{\{0\}}(L)$, therefore
$G_{\{0\}}(L)= G_{\{0\}}(L_1)$.

 $(b)$ Now, let $a\in V\left(G_{\{0\}}(L_1)\right)$. If $x\in  V\left(G_{\{0\}}(L)\right)$, then there exists a nonzero element
 $y\in L$ such that $x\wedge y=0$.
 This implies that at most one of $x$ and $y$ may be in $L_2$, otherwise $a\leq x\wedge y=0$, a contradiction. If $x,y\in L_1$, then
 $x\in V\left(G_{\{0\}}(L_1)\right)$. Without loss of generality, let $x\in L_1$ and $y\in L_2$, which gives $x\wedge a=0$, since $a\leq y$.
 Therefore $x\in G_a$.\\
 Thus $V\left(G_{\{0\}}(L)\right)\subseteq V\Big(G_{\{0\}}(L_1)\cup \big(G_a + N(L_2)\big)\Big)$.

 Since  $ V\Big(G_{\{0\}}(L_1)\cup \big(G_a + N(L_2)\big)\Big)\subseteq V\left(G_{\{0\}}(L)\right)$, the equality holds.

Let $x$ and $y$ be adjacent in $G_{\{0\}}(L)$. Hence at most one of $x$ and $y$ may be in $L_2$. If $x,y\in L_1$,
then $x$ and $y$ are adjacent in $G_{\{0\}}(L_1)$.  Now, without loss
of generality, assume that $x\in L_1$ and $y\in L_2$. Therefore $x\wedge
a=0$. Hence $x\in G_a$, \textit{i.e, } $x$ and $y$ are adjacent in $G_a + N(L_2)$.

Now, let $x$ and $y$ be adjacent in  $G_{\{0\}}(L_1)\cup
 \big(G_a + N(L_2)\big)$. If $x, y\in G_{\{0\}}(L_1)$, we are through. Let $x\in G_a$ and $y\in L_2$. Then $x\wedge a=0$. We claim that $x\wedge y=0$.
 Suppose $x\wedge y\neq 0$. Then we have two possibilities either $x\wedge y\in L_1$ or $x\wedge y\in L_2$. If $x\wedge y\in L_2$,
 then by the definition of adjunct, we have $a\leq x\wedge y$ which yields $a\leq x\wedge y\wedge a=0$, a contradiction.
 Thus $x\wedge y\notin L_2$. Hence $x\wedge y \in L_1$. Since $x\wedge y\leq y$ for $y\in L_2$, again by the definition of adjunct,
 we have $x\wedge y\leq a$. This gives $x\wedge y\wedge x\leq a\wedge x=0$, a contradiction to the fact that $x\wedge y \neq 0$.
 Thus we conclude that $x\wedge a=0$ if and only if $x\wedge y=0$ for any $y\in L_2$. Therefore $G_{\{0\}}(L)= G_{\{0\}}(L_1)\cup \big(G_a + N(L_2)\big)$.

$(c)$ Assume that $a=0$. Let $x\in
V\left(G_{\{0\}}(L)\right)$. Then there exists a nonzero element
$y\in L$ such that $x\wedge y=0$. Hence at most one of $x$ and $y$ may be in
$L_2$. If $x,y\in L_1$, then $x\in G_{\{0\}}(L_1)$. Without loss
of generality, let $x\in L_1$ and $y\in L_2$. If $x\geq b$, then by
the definition of adjunct, we have $x\geq y$,
a contradiction to the fact that $x\wedge y=0$. Hence $x\ngeq b$, \textit{i.e.}, $x\in N(L_1^*
\backslash [b))$. Therefore $x\in V\Big( G_{\{0\}}(L_1) \cup
\big(N(L_1^* \backslash [b)) + N(L_2)\big)\Big)$.

Now, assume that $x\in
V\Big( G_{\{0\}}(L_1) \cup \big(N(L_1^* \backslash [b)) +
N(L_2)\big)\Big)$.

If $x\in G_{\{0\}}(L_1)$, then we are through.
Let $x\in V\Big( N\big(L_1^* \backslash [b)\big) + N(L_2)\Big)$.
If $x\in N\big(L_1^* \backslash [b)\big)$, then $x\ngeq b$. For any
$y\in L_2$, we have $x||y$ and by the definition of adjunct and
$a=0$, we have $x\wedge y=0$. Therefore $x\in
V\left(G_{\{0\}}(L)\right)$, as $y$ is a nonzero element of $L$.

If
$x\in N(L_2)$, then for any atom $p$ of $L_1$, $p\wedge x=0$.
Therefore $x\in G_{\{0\}}(L)$. Hence we have
$V\left(G_{\{0\}}(L)\right)=V\Big( G_{\{0\}}(L_1) \cup \big(N(L_1^* \backslash [b)) +
N(L_2)\big)\Big)$.

 Let $x$ and $y$ be adjacent in  $G_{\{0\}}(L)$. Then at most one of $x$ and $y$ may be in $L_2$.
 If $x,y\in L_1$, then they are adjacent in $G_{\{0\}}(L_1)$. Therefore they are adjacent in
 $G_{\{0\}}(L_1)\cup \Big( N\big(L_1^* \backslash [b)\big) + N(L_2)\Big)$. Without loss of generality, let $x\in L_1$
 and $y\in L_2$. As above, $x\ngeq b$. Hence $x\in ( N(L_1^* \backslash [b))$ and $y\in L_2$.
 Therefore they are adjacent in $N(L_1^* \backslash [b)) + N(L_2)$ and hence in $G_{\{0\}}(L_1)\cup \Big( N\big(L_1^* \backslash [b)\big) + N(L_2)\Big)$.

Conversely, suppose that $x $ and $y$ are adjacent in
$G_{\{0\}}(L_1)\cup \Big( N\big(L_1^* \backslash [b)\big) +
N(L_2)\Big)$. If both $x, y \in G_{\{0\}}(L_1)$, we are done.
Also, at most one of $x$ and $y$ may be in $L_2$. Without loss of
generality, let $x\in L_1$ and $y\in L_2$, then $x\ngeq b$,
\textit{i.e.}, $x\in  N(L_1^* \backslash [b))$. Therefore $x\wedge
y=0$ in $L$. Hence $x$ and $y$ are adjacent in $G_{\{0\}}(L)$.
Therefore we get $G_{\{0\}}(L)= G_{\{0\}}(L_1)\cup \Big(
N\big(L_1^* \backslash [b)\big) + N(L_2)\Big)$.
\end{proof}

The following Figures \ref{f3}, \ref{f4} and \ref{f5} illustrate
Theorem \ref{t2}.

\begin{figure}[h]
\begin{tikzpicture}
\draw (0,0)--(-.5,1)--(0,2)--(0,5); \draw (0,0)--(.5,1)--(0,2);
\draw (0,2)--(2,2.5)--(2,3.5)--(0,4); \draw[fill=white](2,2.5)
circle(.08); \draw[fill=white](2,3.5) circle(.08);
\draw[fill=white](0,2) circle(.08); \draw[fill=white](0,2)
circle(.08); \draw[fill=white](0,3) circle(.08);
\draw[fill=white](0,4) circle(.08); \draw[fill=white](0,5)
circle(.08); \draw[fill=white](-.5,1) circle(.08);
\draw[fill=white](.5,1) circle(.08); \draw[fill=white](0,0)
circle(.08); \node [left] at (0,2){$a$}; \node [below] at
(2,2.5){$y_1$}; \node [above] at (2,3.5){$y_2$}; \node [above] at
(-.6,1){$x_1$}; \node [above] at (.6,1){$x_2$}; \node [left] at
(0,3){$x_3$}; \node [left] at (0,4){$b$}; \node [left] at
(0,5){$1$}; \node [left] at (0,0){$0$}; \draw (0,2.6) ellipse (1.1
and 3);

\draw (2,3) ellipse (.4 and 1.1);

\node [left] at (2.5,1.6){$L_2$}; \node [below] at (0,-.3){$L_1$};
\draw (6,3)--(8,3); \draw[fill=black](6,3) circle(.08);
\draw[fill=black](8,3) circle(.08); \node [left] at (6,3) {$x_1$};
\node [right] at (8,3) {$x_2$};\node [below] at (7,2)
{$G_{\{0\}}(L_1]_a^bL_2)\cong G_{\{0\}}(L_1) $};
\end{tikzpicture}
  \caption{Theorem \ref{t2}, Case $(a)$}\label{f3}
 \end{figure}
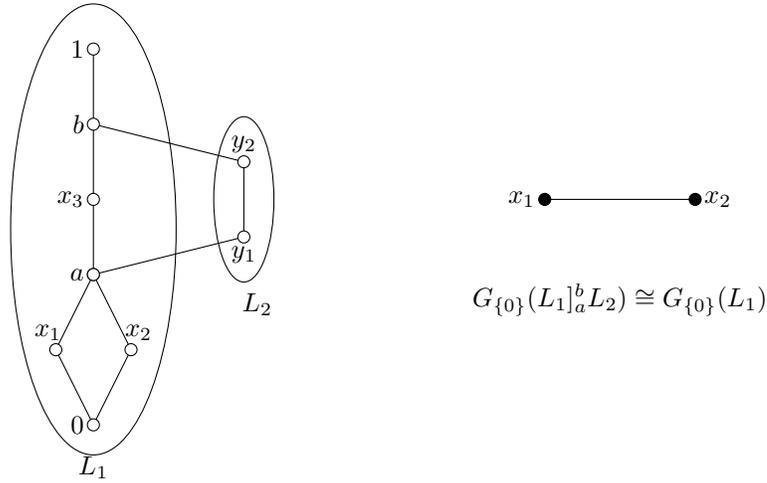
\begin{figure}[h]
\begin{tikzpicture}
\draw (0,0)--(0,4);\draw (0,0)--(-1,1)--(-1,3)--(0,4); \draw
((0,1)--(1,1.5)--(1,2.5)--(0,3); \draw [fill=white](0,0)
circle(.06);\draw [fill=white](0,1) circle(.06);\draw
[fill=white](0,2) circle(.06);\draw [fill=white](0,3) circle(.06);
\draw [fill=white](0,4) circle(.06);\draw [fill=white](-1,1)
circle(.06);\draw [fill=white](-1,3) circle(.06);\draw
[fill=white](1,1.5) circle(.06); \draw [fill=white](1,2.5)
circle(.06); \node [left] at (0,0){$0$};\node [below] at (1,1.5)
{$y_1$};\node [above] at (1,2.5) {$y_2$}; \node [left] at
(0,-.9){$L_1$};\node [right] at (.8,.6) {$L_2$}; \node [below] at
(-1,1){$x_1$}; \node [left] at (0,2){$x_2$}; \node [above] at
(-1,3.1){$x_3$}; \node [left] at (0,4){$1$}; \node [left] at
(0,3){$b$}; \node [left] at (0,1){$a$}; \draw (-.5,2) ellipse (1
and 2.6);\draw (1,2) ellipse (.35 and 1.1); \draw
(3.5,2.5)--(2.5,1)--(3.5,1.5)--(2.5,2)--(3.5,2.5)--(2.5,3)--(3.5,1.5);
\draw [fill=black](2.5,1) circle(.06);\draw [fill=black](2.5,2)
circle(.06);\draw [fill=black](2.5,3) circle(.06); \draw
[fill=black](3.5,1.5) circle(.06);\draw [fill=black](3.5,2.5)
circle(.06); \node [above] at (3.5,2.5){$x_1$}; \node [below] at
(2.5,1){$x_2$}; \node [below] at (3.5,1.5){$x_3$}; \node [above]
at (2.5,2){$a$}; \node [below] at (3,.5) {\tiny $G_{\{0\}}(L_1)$};
\node [above] at (2.5,3) {$b$};

\draw (5,1.5)--(6,2.5)--(5,2.5)--(6,1.5)--cycle; \node [below] at
(5,1.5) {$x_3$}; \node [above] at (5,2.5) {$x_1$}; \node [below]
at (6,1.5) {$y_2$};\node [above] at(6,2.5)
{$y_1$};\draw[fill=black] (5,1.5) circle (.06); \draw[fill=black]
(6,1.5) circle (.06); \draw[fill=black] (5,2.5) circle (.06);
\draw[fill=black] (6,2.5) circle (.06); \draw (5,2) ellipse (.35
and 1.1);\draw (6,2) ellipse (.35 and 1.1);\node [below] at
(5.5,.5) {\tiny $G_a+N(L_2)$};

\draw (10,2.5)--(9,1)--(10,1.5)--(9,2)--(10,2.5)--(9,3)--(10,1.5);
\draw [fill=black](9,1) circle(.06);\draw [fill=black](9,2)
circle(.06);\draw [fill=black](9,3) circle(.06); \draw
[fill=black](10,1.5) circle(.06);\draw [fill=black](10,2.5)
circle(.06); \node [above] at (10,2.5){$x_1$}; \node [left] at
(9,1){$x_2$}; \node [below] at (10,1.5){$x_3$}; \node [left] at
(9,2){$a$}; \node [left] at (9,3) {$b$}; \draw
(10,1.5)--(11,2.5)--(10,2.5)--(11,1.5)--cycle;\draw [fill=black]
(11,1.5) circle(.06);\draw [fill=black] (11,2.5) circle(.06); \node
[below] at (11,1.5) {$y_2$};\node [above] at (11,2.5) {$y_1$}; \node
[below] at (10.5,.5) {\tiny $G_{\{0\}}(L_1]_a^bL_2)= G_{\{0\}}(L_1)\cup
\big(G_a + N(L_2)\big)$};
\end{tikzpicture}
 \caption{Theorem \ref{t2}, Case $(b)$}\label{f4}
\end{figure}
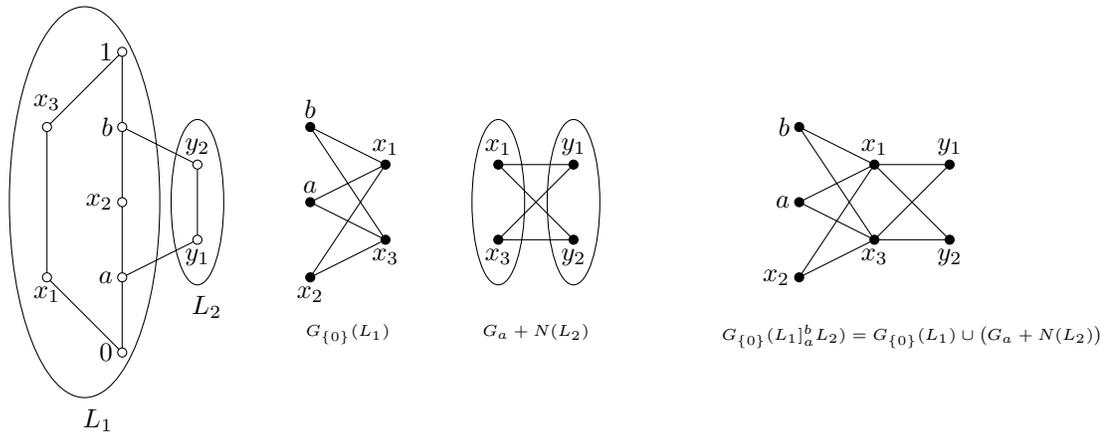
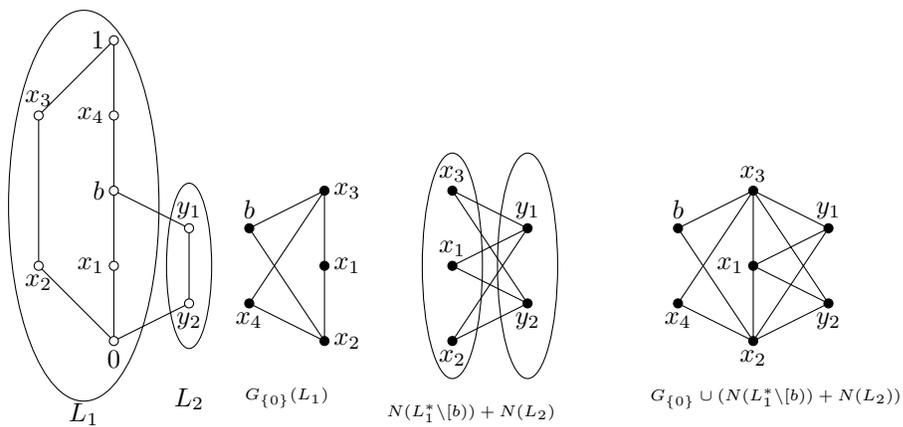
\begin{figure}[h]
\begin{tikzpicture}
\draw (0,0)--(0,4)--(-1,3)--(-1,1)--cycle;\draw
(0,2)--(1,1.5)--(1,.5)--(0,0);\draw [fill=white] (0,0)
circle(.06); \draw [fill=white] (0,1) circle(.06);\draw
[fill=white] (0,3) circle(.06);\draw [fill=white] (0,4)
circle(.06); \draw [fill=white] (-1,3) circle(.06);\draw
[fill=white] (-1,1) circle(.06);\draw [fill=white] (1,.5)
circle(.06);\draw [fill=white] (1,1.5) circle(.06);\draw
[fill=white] (0,2) circle(.06); \node [below] at (0,0) {$0$};\node
[left] at (0,4) {$1$};\node [left] at (0,1) {$x_1$};\node [left]
at (0,2) {$b$};\node [left] at (0,3) {$x_4$};\node [above] at
(1,1.5) {$y_1$};\node [below] at (1,.5) {$y_2$};\node [below] at
(-1,1) {$x_2$};\node [above] at (-1,3) {$x_3$}; \draw (-.4,1.8)
ellipse (1 and 2.6); \draw (1,1) ellipse (.3 and 1.1); \node
[below] at (-.4,-.7) {$L_1$};\node [below] at (1,-.5) {$L_2$};
\draw (2.8,0)--(2.8,2)--(1.8,1.5)--cycle;\draw (2.8,0)--
(1.8,.5)--(2.8,2); \draw [fill=black](2.8,0) circle(.06);\draw
[fill=black](2.8,1) circle(.06);\draw [fill=black](2.8,2)
circle(.06);\draw [fill=black](1.8,0.5) circle(.06);\draw
[fill=black](1.8,1.5) circle(.06);\node [below] at
(1.8,.5){$x_4$};\node [above] at (1.8,1.5){$b$};\node [right] at
(2.8,0){$x_2$};\node [right] at (2.8,1){$x_1$};\node [right] at
(2.8,2){$x_3$};\node [below] at (2.3,-.5){\tiny $G_{\{0\}}(L_1)$};
\draw [fill=black](4.5,2) circle(.06);\draw [fill=black](4.5,1)
circle(.06);\draw [fill=black](4.5,0) circle(.06);\node [below] at
(4.5,0){$x_2$}; \node [above] at (4.5,1){$x_1$};\node [above] at
(4.5,2){$x_3$};\draw [fill=black] (5.5,.5) circle(.06);\draw
[fill=black] (5.5,1.5) circle(.06); \node [above] at
(5.5,1.5){$y_1$};\node [below] at (5.5,.5){$y_2$}; \draw
(5.5,.5)--(4.5,0)--(5.5,1.5)--(4.5,2)--(5.5,.5)--(4.5,1)--(5.5,1.5);\draw
(4.5,1) ellipse (.4 and 1.5);\draw (5.5,1) ellipse (.4 and 1.5);
\node [above] at (4.75,-1.2){\tiny
$N(L_1^*\backslash[b))+N(L_2)$};

\draw (8.5,0)--(8.5,2)--(7.5,1.5)--cycle;\draw (8.5,0)--
(7.5,.5)--(8.5,2); \draw [fill=black](8.5,0) circle(.06);\draw
[fill=black](8.5,1) circle(.06);\draw [fill=black](8.5,2)
circle(.06);\draw [fill=black](7.5,0.5) circle(.06);\draw
[fill=black](7.5,1.5) circle(.06);\node [below] at
(7.5,.5){$x_4$};\node [above] at (7.5,1.5){$b$};\node [below] at
(8.5,0){$x_2$};\node [left] at (8.5,1){$x_1$};\node [above] at
(8.5,2){$x_3$}; \draw
(9.5,.5)--(8.5,0)--(9.5,1.5)--(8.5,2)--(9.5,.5)--(8.5,1)--(9.5,1.5);
\draw [fill=black](9.5,1.5) circle(.06);\draw [fill=black](9.5,.5)
circle(.06); \node [below] at (9.5,.5){$y_2$};\node [above] at
(9.5,1.5){$y_1$};\node [below] at (8.8,-.5){\tiny $G_{\{0\}}\cup
(N(L_1^*\backslash[b))+N(L_2))$};
\end{tikzpicture}

 \caption{Theorem \ref{t2}, Case $(c)$}\label{f5}
\end{figure}

Let  $G$ be a graph and $x, y$ be distinct vertices in $G$. A path
is a simple graph whose vertices can be ordered so that two
vertices are adjacent if and only if they are consecutive in the
list. If $G$ has a $x$, $y$-path, then the {\it distance} from $x$
to $y$, written $d(x,y)$, is the least length of a $x$, $y$-path.
If $G$ has no such path, then $d(x,y)=\infty$. The {\it diameter
}($diam(G)$) is $max_{x,y\in V(G)}d(x,y)$; see West \cite{6}.

\begin{cor}\label{c1}Let $L$ be an adjunct of two chains $C_1$ and $C_2$ with an adjunct pair $(a,1)$,
\textit{i.e.},  $L=C_1]_a ^1 C_2$. If $G_{\{0\}}(L)\neq\emptyset$,
then $a=0$ and  $G_{\{0\}}(L)$ is a complete bipartite graph.
Hence $diam (G_{\{0\}}(L))\leq 2$. Moreover, $G_{\{0\}}(L)=
K_{m,n}$, if $|C_1|=n+2$ and $|C_2|=m$ where $m,n\in
\mathbb{N}$.\end{cor}
\begin{proof} Let $L=C_1]_a ^1 C_2$ and  $G_{\{0\}}(L)\neq\emptyset$. Clearly $V(G_{\{0\}}(C_1))=\emptyset$. If $a\neq 0$, then by Theorem
\ref{t2}($a$), we have $G_{\{0\}}(L)=G_{\{0\}}(C_1)=\emptyset$, a contradiction. Therefore $a=0$.\\
  Now, every element of $C_1\backslash \{0,1\}$ is adjacent to each element of $C_2$. Hence $G_{\{0\}}(L)$ is
a complete bipartite graph, in fact $G_{\{0\}}(L)= K_{m,n}$ whenever
$|C_1|=n+2$ and $|C_2|=m$ where $m,n\in \mathbb{N}$.\end{proof}
\begin{nt} Let $\mathcal{M}_n=\{0,1,a_1,a_2,\cdots,a_n\}$ be a lattice such that $0<a_i<1$, for every  $i$,
 $i=1,2,\cdots,n$ with $a_i\wedge a_j=0$ and $a_i\vee a_j=1$ for
every $ i\neq j$.
\end{nt}
\begin{rem}If $L$ is an adjunct of more than two chains, then $G_{\{0\}}(L)$
need not be bipartite. Consider the lattice $L=\mathcal{M}_3$ depicted in Figure \ref{f6}. Consider $L=C_1]_0 ^1 C_2]_0 ^1 C_3$
where $C_1=\{0,x,1\}$, $C_2=\{y\}$ and $C_3=\{z\}$.  Then
$G_{\{0\}}(L)\cong K_3$, a non bipartite graph.\end{rem}
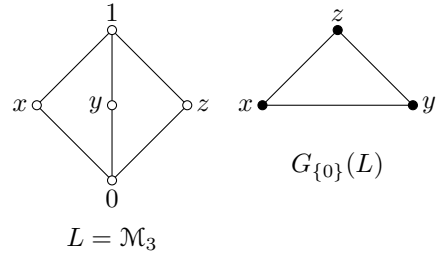
\begin{figure}[h]
\begin{tikzpicture}
\draw (1,0)--(1,2)--(0,1)--(1,0)--(2,1)--(1,2);\draw [fill=white]
(1,0) circle (.06);\draw [fill=white] (1,1) circle (.06);\draw
[fill=white] (1,2) circle (.06);\draw [fill=white] (0,1) circle
(.06);\draw [fill=white] (2,1) circle (.06);\node [above] at (1,2)
{$1$};\node [left] at (1,1) {$y$};\node [left] at (0,1)
{$x$};\node [right] at (2,1) {$z$};\node [below] at (1,0)
{$0$};\node [below] at (1,-.5) {$L=\mathcal{M}_3$}; \draw
(3,1)--(4,2)--(5,1)--(3,1);\draw [fill=black] (3,1) circle
(.06);\draw [fill=black] (4,2) circle (.06);\draw [fill=black]
(5,1) circle (.06);\node [left] at (3,1) {$x$};\node [above] at
(4,2) {$z$};\node [right] at (5,1) {$y$};\node [below] at (4,.5)
{$G_{\{0\}}(L)$};
\end{tikzpicture}
\caption{A dismantlable lattice with non-bipartite zero-divisor
graph}\label{f6}
\end{figure}

The concept of  dismantlable lattice was introduced by  Rival
\cite{17}.
\begin{defn}[Rival
\cite{17}]A finite lattice $L$ having $n$ elements is called {\it dismantlable},
if there exists a chain $L_1 \subset L_2 \subset \cdots \subset
L_n (= L)$ of sublattices of $L$ such that $| L_i| = i$, for all
$i$.\end{defn} The following structure theorem is due to Thakare,
Pawar and Waphare  \cite{18}.
\begin{st}[Theorem 2.2,  Thakare et at. \cite{18}]\label{st}A finite lattice is dismantlable if and only if it is an
adjunct of chains.\end{st}

From the above structure theorem and Corollary \ref{c1}, it is clear that if
the vertex set of a zero-divisor graph of adjunct of two chains is
nonempty then it is a lower dismantlable lattice in the following
sense.
\begin{defn}We call a dismantlable lattice $L$ to be a \textit{lower dismantlable} if it is a chain or
every adjunct pair in $L$ is of the form $(0,b)$ for some $b\in L$.
\end{defn}

It should be noted that any lattice of the form
$L=C_0]_0^{x_1}C_1]_0^{x_2}\cdots]_0^{x_n}C_n$ is always a lower
dismantlable lattice, where $C_i$'s are chains. Consider the
lattices depicted in Figure \ref{f7}. Observe that the lattice $L$
is lower dismantlable whereas the lattice $L'$ is not lower
dismantlable.
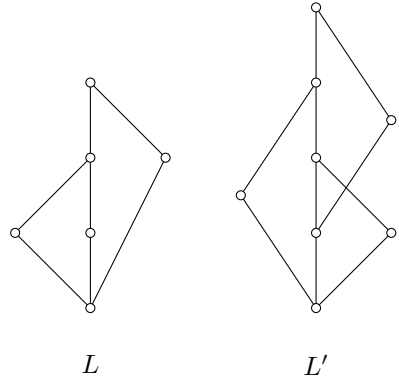
\begin{figure}[h]
\begin{tikzpicture}
\draw (1,0)--(1,3)--(2,2)--(1,0)--(0,1)--(1,2);\draw [fill=white]
(1,0) circle (.06); \draw [fill=white] (1,1) circle (.06);\draw
[fill=white] (1,2) circle (.06);\draw [fill=white] (1,3) circle
(.06);\draw [fill=white] (0,1) circle (.06);\draw [fill=white]
(2,2) circle (.06);\node [below] at (1,-.5) {$L$};
\draw(4,0)--(4,4)--(5,2.5)--(4,1);
\draw(4,3)--(3,1.5)--(4,0)--(5,1)--(4,2);\draw [fill=white] (4,0)
circle (.06);\draw [fill=white] (4,1) circle (.06);\draw
[fill=white] (4,2) circle (.06);\draw [fill=white] (4,3) circle
(.06);\draw [fill=white] (4,4) circle (.06);\draw [fill=white]
(5,2.5) circle (.06);\draw [fill=white] (3,1.5) circle (.06);\draw
[fill=white] (5,1) circle (.06);\node [below] at (4,-.5) {$L'$};
\end{tikzpicture}
\caption{Examples of  lower dismantlable  and non- lower
dismantlable lattice}\label{f7}
\end{figure}
\begin{defn}
An element $x$ in a lattice $L$ is \textit{join-reducible $($meet-reducible$)$} in $L$
 if there exist $y,z\in L$ both distinct from $x$, such that $y\vee z=x$ $(y\wedge z= x)$; $x$ is
 \textit{join-irreducible $($meet-irreducible$)$} if it is not join-reducible $($meet-reducible$)$;
 $x$ is \textit{doubly irreducible} if it is both join-irreducible and meet-irreducible.
 Therefore, an element $x$ is doubly reducible in a lattice $L$ if and only if $x$ has at most one
 lower cover or $x$ has at most one upper cover. The set of all meet-irreducible $($join-irreducible$)$
 elements in $L$ is denoted by \textit{$M(L)$ $(J(L))$}. The set of all doubly irreducible elements in $L$ is
 denoted by $Irr(L)$ and its  in $L$ is denoted by \textit{$Red(L)$}. Thus, if $x\in Red(L)$ then $x$
 is either join-reducible or meet-reducible.

 For an integer $n\geq 3$, a {\it crown} is a partially ordered set $\{x_1, y_1,x_2,y_2,\cdots, x_n, y_n\}$
 in which $x_i\leq y_i$, $x_{i+1}\leq y_i$, for $i=1,2,\cdots, n-1$ and $x_1\leq y_n$ are the
 only comparability relations (see Figure \ref{f8}).\end{defn}
 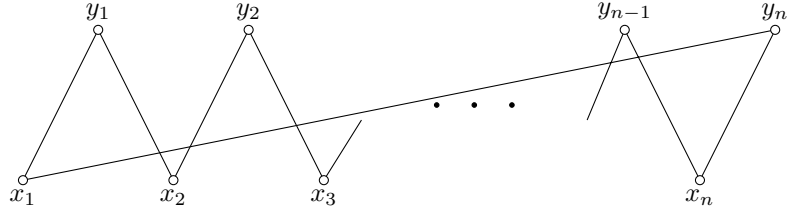
\begin{figure}[h]
 \begin{tikzpicture}
 \draw (0,0)--(1,2)--(2,0)--(3,2)--(4,0)--(4.5,.8);
 \draw (7.5,.8)--(8,2)--(9,0)--(10,2)--(0,0);
 \draw [fill=white](0,0) circle (.06);\draw [fill=white](1,2) circle (.06);\draw [fill=white](2,0) circle (.06);\draw [fill=white](3,2) circle (.06);
 \draw [fill=white](4,0) circle (.06);\draw [fill=white](8,2) circle (.06);\draw [fill=white](9,0) circle (.06);\draw [fill=white](10,2) circle (.06);
 \node [below] at (0,0) {$x_1$};\node [below] at (2,0) {$x_2$};\node [below] at (4,0) {$x_3$};\node [below] at (9,0) {$x_n$};
 \node [above] at (1,2) {$y_1$};\node [above] at (3,2) {$y_2$};\node [above] at (8,2) {$y_{n-1}$};\node [above] at (10,2) {$y_n$};

 \draw [fill=black]
 (5.5,1) circle (.03); \draw [fill=black] (6,1) circle (.03);\draw
 [fill=black] (6.5,1) circle (.03);
 \end{tikzpicture}
 \caption{Crown of order $2n$}\label{f8}
 \end{figure}

Note that if $L$ is lower dismantlable lattice with the greatest element 1 as a join-reducible element, then it is easy to observe that
every nonzero nonunit element of $L$ is a vertex of $G_{\{0\}}(L)$.

 The following lemma gives the properties of lower dismantlable
 lattices which will be used in the sequel frequently.
\begin{lem}\label{l1}Let $L=C_0]_0^{x_1}C_1]_0^{x_2}\cdots]_0^{x_n}C_n$ be a lower dismantlable lattice,
where $C_i$'s are chains. Then for nonzero elements $a,b\in L$, we have.
\begin{enumerate}
\item[$i)$] $a\wedge b=0$ if and only if $a||b$ $($where $a||b$ means $a$ and $b$ are incomparable$)$.
\item[$ii)$] Let $a\in C_i$ with $b\notin C_i$, then $a\leq b$ if and
only if $x_i\leq b$.
\item[$iii)$] If $(0,1)$ is an adjunct pair $(${\it i.e.}, $x_i=1$ for
some $i\in \{1,2,\cdots, n\})$, then
$|V\left(G_{\{0\}}(L)\right)|=|L|-2$.
\end{enumerate}\end{lem}
\begin{proof}$i)$Suppose $a||b$ and $a\wedge b\neq
0$. It is clear that there is an adjunct pair $(a_1,b_1)$ in the
adjunct representation of $L$ such that $a_1=a\wedge b\neq 0$, a
contradiction to the definition of lower
dismantlability of $L$. The converse is obvious.\\
$ii)$ Let $a\in C_i$ and  $b\notin C_i$ with $x_i\leq b$. As $C_i$ is joined at $x_i$,
we must have $a\leq x_i$. Hence $a\leq b$. Conversely, suppose
$a\leq b$. Now, $a\in C_i$ and $b\notin C_i$. If $x_i\nleq b$, we
have either $b\leq x_i$ or $x_i||b$. The second case is
impossible, by $(i)$ above, as it gives $a\wedge b=0$, since $a\leq x_i$. Also, it is given that
$b\notin C_i$, hence $a||b $ which yields $a\wedge b=0$, a contradiction to the fact that $a\leq b$ and $a\neq 0$. Therefore
$x_i\leq b$.\\
$iii)$ As $L$ is a lower dismantlable lattice having $(0,1)$ as an
adjunct pair, $L$ contains at least two chains in its adjunct
representation. Also $a\wedge b=0$ if and only if $a||b$. Hence
any $a\in L\backslash \{0,1\}$ is in $V\left(G_{\{0\}}(L)\right)$.
and consequently $|V\left(G_{\{0\}}(L)\right)|=|L|-2$.
\end{proof}

 Now, we recall some definitions
from graph theory.
\begin{defn}A {\it cycle} is a graph with an equal number of
vertices and edges whose vertices can be placed around a circle so
that two vertices are adjacent if and only if they appear
consecutively along the circle. The {\it girth } of of a graph
with cycle, written $gr(G)$, is the length of its shortest cycle.
A graph with no cycle has infinite  grith. A graph with no cycle
is \textit{acyclic}. A \textit{tree} is a connected acyclic graph.
A tree is called a \textit{rooted tree} if one vertex has been
designated the root, in which case the edges have a natural
orientation, towards or away from the root. A vertex $w$ of a
rooted tree is called an \textit{ancestor} of $v$ if $w$ is on the
unique path from $v$ to the root of the tree; see West \cite{6}.

 Let $T$ be a rooted tree with the root $R$ has at least
 two branches. Let $G(T)$ be the {\it non-ancestor graph} of $T$, \textit{i.e.}, $V(G(T))=T\backslash \{R\}$ and two vertices
 are adjacent if and only if no one is an ancestor of the other. Denote the class of non-ancestor graphs of rooted trees by $\mathcal{G_T}$.
 The \textit{cover graph} of a lattice $L$, denoted by
  $CG(L)$, is the graph whose vertices are the elements of $L$ and whose edges are the pairs $(x,y)$ with $x,y\in L$
  satisfying $x\prec y$ or $y\prec x$. The \textit{comparability graph} of a lattice $L$, denoted by $C(L)$, is the graph whose
  vertices are the elements of $L$ and two vertices $x$ and $y$ are adjacent if and only if $x$ and $y$ are comparable.
  The complement of the comparability graph $C(L)$, {\it i.e.}, $C(L)^c$, is called the \textit{incomparability graph} of $L$.\end{defn}

The following result is due to Kelly and Rival
\cite{10}.
\begin{thm}[Theorem 3.1, Kelly and Rival
\cite{10}]\label{t3}A finite lattice is dismantlable if and only if it contains no crown.\end{thm}

In the following theorem, we characterize the zero-divisor graph
$G_{\{0\}}(L)$ of a lower dismantlable lattice $L$ in terms of the
cover graph $CG(L)$ and the incomparability graph $C(L)^c$.
\begin{thm}\label{t4} The following statements are equivalent for a finite lattice $L$ with 1 as a join-reducible element.
\begin{enumerate}
\item[$(a)$] $L$ is a lower dismantlable lattice. \item[$(b)$]
Every nonzero element of $L$ is a meet-irreducible element.
\item[$(c)$] The cover graph $CG(L\backslash \{0\})$ of $L\backslash \{0\}$ is a tree.\item[$(d)$] The zero-divisor
graph of $L$ coinsides with the incomparability graph of $L\backslash
\{0,1\}$.
\end{enumerate}
\end{thm}
\begin{proof} $(a)\Rightarrow (b)$ Let $L$ be a lower dismantlable lattice. Then
by Structure Theorem,\\ $L= C_1]_{0}^{a_1}
C_2]_{0}^{a_2}\cdots]_{0}^{a_{n-1}} C_n$, where each $C_i$ is a chain.
Let $a(\neq 0)\in L$ be an element which is not meet-irreducible.
Then $a=b\wedge c$ for some $b,c\neq a$. But then $b||c$. By Lemma
\ref{l1}, $a=b\wedge c=0$, a contradiction to $a\neq 0$. Hence every
nonzero element of $L$ is meet
irreducible element.\\
$(b)\Rightarrow (c)$ Let $CG(L\backslash \{0\})$ be the cover graph of $L\backslash \{0\}$ and let
$C: a_1-a_2-\cdots-a_n-a_1$ be a cycle in $CG(L\backslash \{0\})$.
For distinct $a_1,a_2,a_3$ we have the following three cases.\\
\textbf{Case $(1)$:} Let $a_1\prec a_2\prec a_3$ be a chain. Then for $a_4$,
we have either $a_1\prec a_2\prec a_3\prec a_4$ is a chain or
$a_1\prec a_2\prec a_3$ and $a_4\prec a_3$. Then, $a_{i+1}\leq a_i
$, for all $i\geq 4$; otherwise we get $a_4$ as a meet reducible
element. Hence $a_n\leq a_{n-1}$ and $a_n=a_{n-1}\wedge a_1\neq 0$,
as $C: a_1-a_2-\cdots-a_n-a_1$ is a cycle. This contradicts the fact
that every nonzero element is meet-irreducible. On the other hand, if $a_1\prec a_2\prec a_3\prec
a_4$ is chain, then using the above arguments, we have $a_1\prec
a_2\prec a_3\prec a_4\prec a_5$ is a chain, Continuing in this way
we get $a_1\prec a_2\prec \cdots\prec a_n$ is a chain and $a_1\leq
a_n$. Hence $a_1$ and $a_n$ can not be adjacent, a contradiction to
the fact that $C$ is a
cycle in $CG(L\backslash \{0\})$.\\
\textbf{Case $(2)$:} $a_2= a_1 \wedge a_3$, which is impossible, as $a_1\wedge a_3=a_2\neq 0$, a
contradiction to lower  dismantlability of $L$.\\
\textbf{Case $(3)$:} $a_2= a_3\vee a_1$. Note that $a_3\nleq a_4$
otherwise $a_3$ is the meet of $a_2$ and $a_4$, a contradiction to the fact that every nonzero element is meet-irreducible.
Hence $a_4\leq a_3$. In fact $a_{m+1}\leq a_m$, for $m\geq 3$. Using
the above arguments, we again obtain a contradiction. Hence
$CG(L\backslash \{0\})$ is a connected acyclic graph. Therefore it is
a tree.\\
$(c)\Rightarrow (a)$ If $L$ contains a crown, then $CG(L\backslash \{0\})$ contains a
cycle, a contradiction. Hence $L$ does not contain a crown. By
applying Theorem \ref{t3}, $L$ is a dismantlable lattice. Now, let
$a$ and $b$ be incomparable elements of $L$. Suppose that $a\wedge
b\neq 0$. Let $a\wedge b=a_1\prec a_2\prec\cdots\prec a_i=a\prec
\cdots\prec a_n=a\vee b$, be a covering and also $a\wedge b=b_1\prec
b_2\prec\cdots\prec b_j=b\prec\cdots\prec b_m=a\vee b$, be another
covering, distinct from the first covering (such coverings exist,
since $a||b$ ). Then $a_1- a_2-\cdots- a_n=b_m-b_{m-1}-\cdots-b_1=a_1$ is
a cycle in $CG(L\backslash \{0\})$, a contradiction to the fact that
$CG(L\backslash \{0\})$
is a tree. Thus $L$ is a lower dismantlable lattice.\\
$(d)\Rightarrow (b)$ Suppose $G_{\{0\}}(L)=C(L\backslash
\{0,1\})^c$, for some lattice $L$. We want to show
that $L$ does not contain any nonzero meet-reducible element. Suppose on
the contrary, $L$ has a nonzero meet-reducible element say $b$. Then there
exist $a, c \in L$ with $a, c \neq b$, such that $b = a\wedge c$.
Thus $a$ and $c$ are incomparable. So there is an edge $a-c$ in
$C(L\backslash \{0,1\})^c$. But $a\wedge c\neq 0$, hence $a$ and $c$
are not adjacent in $G_{\{0\}}(L)$, a contradiction. Consequently every
nonzero element of $L$ is a meet-irreducible element.\\
$(b)\Rightarrow (d)$ Suppose every nonzero element of $L$ is meet
irreducible. By the equivalence of $(a)$ and $(b)$ above,
$L$ is a lower dismantlable lattice. Hence $a\wedge b=0$ if and only if
$a||b$. Therefore $G_{\{0\}}(L)=C(L\backslash \{0,1\})^c$.\end{proof}

Note that a result similar to the equivalence of statements $(b)$ and $(d)$
 of Theorem \ref{t4} can be found in Survase \cite{22}.

Grillet and Varlet \cite{19} introduced the concept of
0-distributive lattices as a generalization of distributive
lattices. A lattice $L$ with 0 is called \emph{ 0-distributive}
if, for every triplet $(a, b, c)$ of elements of $L$, $a\wedge b =
a\wedge c = 0$ implies $a\wedge (b\vee c) = 0$. More details about
0-distributive posets can be found in Joshi and Waphare \cite{20}.
Forbidden configurations for 0-distributive lattices are
obtained by Joshi \cite{21}.
\begin{lem}\label{l2}If $L$ is an adjunct of two chains with $(0,1)$ as an
adjunct pair, then $L$ is 0-distributive.
\end{lem}
\begin{proof}Let  $a\wedge b =
a\wedge c = 0$, for $a,b,c \in L$. By Lemma \ref{l1}, $b$ and $c$
are either comparable or $b\wedge c=0$. If $b\wedge c=0$, this
together with $a\wedge b = a\wedge c = 0$ gives $L$ as adjunct of at
least three chains, a contradiction. Hence $b$ and $c$ are comparable. By $a\wedge b=a\wedge c=0$, we have $a\wedge (b\vee c)=0$.
\end{proof}
The following result is due to Joshi \cite{7}.
\begin{thm}[Theorem 2.14, Joshi \cite{7}]\label{t5}Let $L$ be a 0-distributive lattice. Then $G_{\{0\}}(L)$ is complete
bipartite if and only if there exist two minimal prime ideals
$P_1$ and $P_2$ of $L$ such that $P_1\cap P_2=\{0\}$.
\end{thm}
\begin{thm}\label{t6}Let $L$ be a lower dismantlable lattice having $(0,1)$ as an adjunct pair. Then the following statements are equivalent.
\begin{enumerate}
\item[$(a)$] $G_{\{0\}}(L)$ is a complete bipartite graph.
\item[$(b)$]$L$ is an adjunct of two chains only. \item[$(c)$] $L$
has exactly two atoms and exactly two dual atoms.\item[$(d)$]
There exist two minimal prime ideals $P_1$ and $P_2$ of $L$ such
that $P_1\cap P_2=\{0\}$. \item[$(e)$] $L$ is a
0-distributive lattice.\end{enumerate}
\end{thm}
\begin{proof}$(a)\Rightarrow (b)$ Suppose $G_{\{0\}}(L)=K_{m,n}$. Note that the independent set of
$K_{m,n}$ forms a chain in $L$. If $L$ is an adjunct of more than
two chains, then $L$ contains  at least three
atoms, which forms a triangle in $G_{\{0\}}(L)$, a contradiction. Hence $L$ is adjunct of two chains only.
 Moreover, as $(0,1)$ is an adjunct pair, $L=C_1]_0^1C_2$, where $C_1$ and $c_2$ are chains.\\
$(b)\Rightarrow (c)$ Obvious.\\
$(c)\Rightarrow (d)$ Let $d_1$ and $d_2$ be only dual atoms of $L$.
Then $P_1=(d_1]=\{x\in L~:~x\leq d_1\}$ and $P_2=(d_2]$ are ideals of $L$. In fact, $P_1\cap P_2=\{0\}$,
otherwise we get a nonzero meet-reducible element
in $L$, a contradiction to Theorem \ref{t4}. We claim that $P_1$ and
$P_2$ are prime ideals. Let $a\wedge b\in P_1$, {\it i.e.}, $a\wedge
b\leq d_1$. If $a\wedge b\neq 0$, then $a$ and $b$ are comparable. Hence $a\in P_1$ or $b\in P_1$. Suppose $a\wedge b=0$ and $a,b\notin
P_1=(d_1]$. Then by Lemma \ref{l1}, $a||b$. Since $d_2$ is the only dual atom other that $d_1$, we have $a,b\leq d_2$, which gives $(0,d_2)$ is an
adjunct pair in $L$,
 hence there exist two atoms below $d_2$, {\it i.e.}, $L$ contains three atoms (one below $d_1$ and two below $d_2$), a contradiction.
 Therefore $P_1$ is a prime ideal. Similarly $P_2$ is a prime ideal.\\
$(d)\Rightarrow (a)$ Follows from Theorem 2.14 \cite{7}.\\
$(a)\Rightarrow (e)$ Let $G_{\{0\}}(L)$ be a complete bipartite graph. By the
equivalence of statements $(a)$ and $(b)$, $L$ is an adjunct of
exactly two chains. By Lemma \ref{l2}, $L$ is
0-distributive.\\
$(e)\Rightarrow (a)$ Suppose that $L$ is 0-distributive. Assume on the
contrary, assume that $L$ is an adjunct of more than two chains,
{\it i.e.}, $L$ contains at least three atoms, say $a,b,c$.
Clearly $a\wedge b=b\wedge c=a\wedge c=0$.
We consider the following two cases.\\
\textbf{Case $(1):$} If $a\vee b= a\vee c=b\vee c$, then $L$ contains a
sublattice isomorphic to $\mathcal{M}_3$ (as shown in Figure
\ref{f6}), a contradiction to
0-distributivity of $L$.\\
\textbf{Case $(2):$} Let $a\vee b$ $||$ $b\vee c$. Clearly $b\leq (a\vee
b)\wedge (b\vee c)$ is a nonzero meet-reducible element in $L$, a
contradiction to Theorem \ref{t4}. Hence $a\vee b$ and $b\vee c$ are
comparable. Without loss of generality, suppose $a\vee b\leq b\vee
c$. Hence $a\leq a\vee b\leq b\vee c$, which gives $a\wedge (b\vee
c)=a$, but $a\wedge b=0$ and $a\wedge c=0$, again a contradiction to
0-distributivity of $L$.

Thus in any case $L$ does not contain three atoms. Since $(0,1)$
as an adjunct pair, it is clear that $L$ is an adjunct of exactly two
chains. Therefore $G_{\{0\}}(L)$ is a complete bipartite
graph.\end{proof}

The following result is due to Joshi \cite{7}.
\begin{thm}[Theorem 2.4, Joshi \cite{7}]\label{t7}Let $L$  be a lattice. Then $G_{\{0\}}(L)$ is connected with
$diam\big(G_{\{0\}}(L)\big) \leq 3$.
\end{thm}
The following result can be found in Alizadeh et al.
\cite{1}.
\begin{thm}[Theorem 4.2, Alizadeh \cite{1}]\label{t8}Let $L$ be a lattice, then $gr(G_{\{0\}}(L)\in
\{3,4,\infty\}$.
\end{thm}

In the following theorem, we characterize the diameter and girth
of $G_{\{0\}}(L)$ for a lower dismantlable lattice $L$.
\begin{thm}\label{t9}Let $L$ be a lower dismantlable lattice which is an adjunct of $n$ chains,
where $n\geq 2$. Then $V((G_{\{0\}}(L)))\neq\emptyset$ and $diam
(G_{\{0\}}(L))\leq 2$. Moreover if $(0,1)$ is the only adjunct pair
in $L$, then $diam (G_{\{0\}}(L))=1$ if and only if $L\cong
\mathcal{M}_n$. Further
\begin{enumerate} \item[$(a)$] $gr (G_{\{0\}}(L))=3$ if and only if
$L$ is an adjunct of at least three chains.\item[$(b)$] $gr
(G_{\{0\}}(L))=4$ if and only if $L=C_1 ] _0^aC_2$ with $|C_2|\geq
2$. \item[$(c)$]$gr (G_{\{0\}}(L))=\infty$ if and only if $L=C_1 ]
_0^aC_2$ with $|C_2|=1$.
\end{enumerate}\end{thm}
\begin{proof}Let $a,b\in V((G_{\{0\}}(L)))$. If $a||b$, then by Lemma \ref{l1}, $a\wedge b=0$ and hence $a,b$ are adjacent. Further let $a$ and $b$ are comparable, say $a\leq
b$. Since $b\in V((G_{\{0\}}(L)))$, there is an
element $c(\neq 0)$ such that $b\wedge c=0$. Hence $a\wedge c=0$.
Thus we get a path $a-c-b$ of length $2$. This shows that $d(a,b)\leq
2$ and in any case $diam (G_{\{0\}}(L))\leq 2$.

Let $(0,1)$ be the only adjunct pair in $L$. If $L\cong
\mathcal{M}_n$, then $G_{\{0\}}(L)\cong K_n$. Hence $diam
(G_{\{0\}}(L))=1$. Conversely, suppose $diam (G_{\{0\}}(L))=1$. If
$L=C_1 ] _0^1C_2$, then   $G_{\{0\}}(L)\neq \emptyset $ if and
only if $|C_1|\geq 3 $ and $|C_2|\geq 1$. If $|C_1|\geq 3 $ and
$|C_2|= 2$ then for any $a,b\in C_2$ we have $d(a,b)=2$. Now, If
$|C_1|\geq 4$ and $|C_2|= 1$ then for any $a,b\in C_1\setminus
{\{0,1\}}$, again $d(a,b)=2$, a contradiction to the fact that
$diam (G_{\{0\}}(L))=1$. Therefore $|C_1|= 3 $ and $|C_2|= 1$ and
the result follows by induction.

By Theorem \ref{t8}, $gr(G_{\{0\}}(L)\in
\{3,4,\infty\}$.\\
$(a)$ If $L$ is an adjunct of at least three chains, then it
contains at least three atoms. Hence $G_{\{0\}}(L)$ contains a
triangle, and $gr (G_{\{0\}}(L))=3$. Conversely, let $gr
(G_{\{0\}}(L))=3$. If $L=C_1 ] _0^aC_2$, then by Corollary
\ref{c1}, $G_{\{0\}}(L)$ is complete bipartite. Therefore it does
not contain an odd cycle, a contradiction. Hence $L$ is
adjunct of at least three chains.\\
$(b)$ and $(c)$ If $L=C_1 ] _0^aC_2$, then by Corollary \ref{c1},
$G_{\{0\}}(L)$ is a complete bipartite graph. Then $|C_2|=1$ or
$|V(G_{\{0\}})(L)\cap C_1|=1$ if and only if $gr
(G_{\{0\}}(L))=\infty$. If $|C_2|\geq 2$ or $|V(G_{\{0\}})(L)\cap C_1|\geq 2$, then $gr
(G_{\{0\}}(L))=4$.
\end{proof}
\begin{rem}Note that if we drop the condition of lower dismantlability of $L$, then $diam(G_{\{0\}}(L))$ may exceed 2.
Consider  a lattice $L=C_1]_0 ^e C_2]_b ^1C_3]_0 ^d C_4$, where $C_1=\{0,a,e,1\}$, $C_2=\{b\}$, $C_3=\{d\}$ and
$C_4=\{c\}$. Then $diam (G_{\{0\}}(L))=3$ as shown in Figure
\ref{f2}.
\end{rem}

 Now, we give a realization of zero-divisor graphs of lower dismantlable lattices, \textit{i.e.},
we describe graphs that are the zero-divisor graphs of lower dismantlable lattices.

\begin{proof}[\bf{Proof of Theorem \ref{t1}}] $(a)\Rightarrow(b)$ Let $G\in \mathcal{G_T}$. Hence $G=V(T\backslash \{R\})$ for some rooted tree $T$
with the root $R$. Let $L=V(G)\cup \{R\}\cup\{0\}$. Define a relation $\leq$
on $L$ by, $a\leq R$, $0\leq a$ and $a\leq a$, for every $ a\in
L$. If $a\neq b$, then $a<b$ if and only if $b$ is an ancestor of
$a$. Clearly, $(L, \leq)$ is a poset. If $a||b$, then no one is
ancestor of the other, hence $0$ is the only element below $a$ and
$b$, \textit{i.e.}, $a\wedge b=0$. Let $A=\{c\in L~|~ c
\textnormal{ is common ancestor of} ~a~ \textnormal{and }~ b
\}$. Then $A\neq \emptyset$, as $R\in A$. We claim that the set
$A$ forms a chain. Let $x,y\in A$ with $x||y$. Hence $x$ and $y$
 are ancestors of $a$ and $b$ both. But then $a-x-b-y-a$ is a cycle in the undirected graph of a rooted tree, a contradiction.
 Thus $A$ is a chain. Then the smallest element of $A$ (it exists due to finiteness of $L$) is nothing but $a\vee b$. Hence $L$ is a lattice with the
 greatest element $R$, now denoted by 1.
 Since meet of any two incomparable elements is zero, $L$ does not contain a crown. Hence by Theorem \ref{t3}, $L$ is a dismantlable lattice,
 say $L= C_0]_{a_1}^{b_1} C_1]_{a_2}^{b_2}\cdots]_{a_n}^{b_n} C_n$. Since meet of any two incomparable elements is zero,
 we get $a_i=0$, $\forall i$. Therefore $L$ is a lower dismantlable lattice having 1 as join-reducible element
(since the root of tree has at least two branches, hence $1$ is join-reducible). Also, $a\wedge b=0$ if and only if no
one is an ancestor of the other. Therefore $G=G_{\{0\}}(L)$.\\
$(b)\Rightarrow(c)$ Follows by Theorem \ref{t4}.\\
$(c)\Rightarrow(a)$ Let $L$ be a lower dismantlable lattice. Let $G$ be an incomparability graph of $L\backslash \{0,1\}$,
{\it i.e.}, $G= C(L\backslash \{0,1\})^c$. Then Theorem \ref{t4}, the
cover graph of $L\backslash \{0\}$ is a rooted tree, say $T$, with a root $1$. Let $H$ be the non-ancestor graph of $T$.
Then clearly, $V(H)=V(G)$ and $a$ and $b$ are adjacent in $G$ if and only if $a||b$
 if and only if no one is ancestor of the other if and only if $a$ and $b$ are adjacent in $H$. Hence $G=H$.\end{proof}

\textbf{Acknowledgements:} The authors are grateful to the referee for fruitful suggestions. The first author is financially supported by University Grant Commission, New Delhi via minor research project File No. 47-884/14(WRO).

 \end{document}